\documentclass[12pt,letterpaper,english]{amsart}
\usepackage{lmodern}
\usepackage{helvet}
\usepackage{hyperref}

\usepackage[T1]{fontenc}
\usepackage[latin9]{inputenc}
\usepackage{mathrsfs}
\usepackage{amsthm}
\usepackage{amstext}
\usepackage{amssymb}
\usepackage{tikz}
\usepackage{esint}
\usepackage{graphicx}
\usepackage{cancel}
\usepackage{enumitem}
\DeclareFontFamily{OT1}{pzc}{}
\DeclareFontShape{OT1}{pzc}{m}{it}{<-> s * [1.10] pzcmi7t}{}
\DeclareMathAlphabet{\mathpzc}{OT1}{pzc}{m}{it}
\usepackage[bbgreekl]{mathbbol}
\DeclareSymbolFontAlphabet{\mathbb}{AMSb}
\DeclareSymbolFontAlphabet{\mathbbl}{bbold}

\makeatletter

\pdfpageheight\paperheight
\pdfpagewidth\paperwidth

\newcommand{\noun}[1]{\textsc{#1}}

\numberwithin{equation}{section}
\numberwithin{figure}{section}
\usepackage{enumitem}		
\theoremstyle{plain}
\newtheorem{thm}{\protect\theoremname}[section]
\theoremstyle{remark}
\newtheorem{rem}[thm]{\protect\remarkname}
\theoremstyle{definition}
\newtheorem{defn}[thm]{\protect\definitionname}
\theoremstyle{definition}
\newtheorem{example}[thm]{\protect\examplename}
\theoremstyle{plain}
\newtheorem{cor}[thm]{\protect\corollaryname}
\theoremstyle{plain}

\theoremstyle{plain}

\usepackage{amsfonts}
\usepackage{amssymb}
\usepackage[all]{xy}
\usepackage{array}
\usepackage{lmodern}
\usepackage[T1]{fontenc}
\usepackage{bm}

\newenvironment{psmatrix}
  {\left(\begin{smallmatrix}}
  {\end{smallmatrix}\right)}

\def\d{\partial}
\def\QQ{\mathbb{Q}}
\def\RR{\mathbb{R}}
\def\CC{\mathbb{C}}

\def\ZZ{\mathbb{Z}}
\def\PP{\mathbb{P}}

\def\GG{\mathbb{G}}

\def\W{\mathcal{W}}

\def\Z{\mathcal{Z}}
\def\V{\mathcal{V}}
\def\Y{\mathcal{Y}}
\def\X{\mathcal{X}}
\def\T{\mathcal{T}}
\def\Y{\mathcal{Y}}

\def\H{\mathcal{H}}

\def\NN{\mathbb{N}}
\def\D{\mathcal{D}}
\def\vf{\varphi}

\def\Res{\text{Res}}
\def\MHS{\mathrm{MHS}}
\def\IH{\mathrm{IH}}

\def\B{\mathcal{B}}
\def\Bb{\overline{\B}}
\def\Xb{\overline{\X}}
\def\E{\mathcal{E}}
\def\SL{\mathrm{SL}}
\def\UH{\mathfrak{H}}
\def\TS{\widetilde{\Sigma}}
\def\tT{\widetilde{\T}}
\def\QB{\overline{\QQ}}
\def\CH{\mathrm{CH}}
\def\uc{\underline{c}}

\def\ANF{\mathrm{ANF}}
\def\Ext{\mathrm{Ext}}
\def\AVMHS{\mathrm{AVMHS}}
\def\Hom{\mathrm{Hom}}
\def\Hg{\mathrm{Hg}}
\def\b{\mathfrak{b}}
\def\c{\mathfrak{c}}
\def\r{\mathfrak{r}}
\def\ay{\mathbf{i}}
\def\G{\mathfrak{G}}
\def\C{\mathcal{C}}
\def\Yb{\overline{\Y}}
\def\Cb{\overline{\C}}
\def\fZ{\mathfrak{Z}}
\def\fW{\mathfrak{W}}
\def\K{\mathcal{K}}
\def\Kb{\overline{\K}}

\def\PSL{\mathrm{PSL}}
\def\g{\mathcal{G}}
\def\PD{\PP_{\Delta}}

\theoremstyle{definition}

\theoremstyle{definition}

\theoremstyle{definition}
\newtheorem*{thx}{Acknowledgments}

\makeatother

\usepackage{babel}
  \providecommand{\corollaryname}{Corollary}
  \providecommand{\definitionname}{Definition}
  \providecommand{\remarkname}{Remark}
\providecommand{\theoremname}{Theorem}
 \providecommand{\examplename}{Example}

\begin{document}

\title{A note on higher Green's functions}

\author{Charles F. Doran}

\author{Matt Kerr}

\begin{abstract}
We give a cycle-theoretic proof of the Gross-Zagier conjecture in weight four for several modular curves of genus zero.
\end{abstract}
\maketitle

\section{Introduction}\label{S1}

Let $\Gamma\leq \PSL_2(\ZZ)$ be a congruence subgroup, $\B:=\Gamma\backslash\UH$ the corresponding modular curve, $\Sigma\subset \B$ a finite subset with preimage $\TS$ under $\rho\colon\UH\twoheadrightarrow\B$, and $k\in\ZZ_{\geq 2}$.  A \emph{higher Green's function} $G(\tau)$ of weight $2k$ and level $\Gamma$ with poles on $\Sigma$ is a $\Gamma$-invariant real-analytic function on $\UH\setminus\TS$ limiting to zero on cusps, asymptotic to $c_{\rho(\hat{\tau})}\log|\tau-\hat{\tau}|$ at each $\hat{\tau}\in \TS$ (for some choice of $\uc=(c_{\sigma})_{\sigma\in \Sigma}\in \QQ^{\Sigma}$), and satisfying $\Delta_{\text{hyp}}G:=-4\Im(\tau)^2\frac{\d}{\d\bar{\tau}}\frac{\d}{\d\tau}G=k(1-k)G$.  These exist (and are unique) for any choice of $\Gamma$, $\Sigma$, $k$ and $\uc$.  We will denote by $G_b(\tau)$ the Green's function associated to $\Sigma=\{b\}$ and $c=1$ (and given $k,\Gamma$).

Denoting by $T_m\in \CH^1(\Bb\times\Bb)$ the Hecke correspondences, a finite sum $\T:=\sum_m a_m T_m$ ($a_m\in \QQ$) with $\T_*S_{2k}(\Gamma)=\{0\}$ is called a \emph{relation} (of level $2k$).  In the final pages of \cite{GrZ} (extended in \cite{GKZ}), Gross and Zagier conjectured that for any CM point $\b\in \B$ and relation $\T$, the combination $G_{\T,\b}:=\sum_m m^{k-1}a_mT_m^*G_{\b}$ takes values in $\QB\log\QB$ on all CM points in the complement of $\Sigma_{\T,\b}:=|\T^*\{\b\}|$, where $G_{\T,\b}$ has poles.  (They are a little more precise about the special values, and assume $\Gamma=\Gamma_0(N)$ for some $N\in \NN$, with $a_m=0$ if $(m,N)\neq 1$.)  After numerous advances by many researchers, this conjecture was proved in 2022 by Bruinier, Li, and Yang \cite{BLY}.

In earlier work of Mellit \cite{Me} and Kerr \cite{Ke}, the Gross-Zagier conjecture was connected to existence questions for certain motivic cohomology classes (of $K_1$-type) on the self-fiber product $\E^{2k-2}$ of the elliptic modular surface $\E\to\B$.  This theme has been picked up in an impressive recent preprint \cite{BF} by Brown and Fonseca, from the standpoint of motives and periods, and features a motivic proof of Gross-Zagier in weight 4 and level $1$ ($\Gamma=\PSL_2(\ZZ)$).  Also see the related results of Sreekantan \cite{Sr1,Sr2}.

Inspired by their work, the purpose of this note is to give a simple account in the weight 4 ($k=2$) case of two things:
\begin{itemize}[leftmargin=0.5cm]
\item Gross-Zagier may be viewed as an immediate consequence of the Beilinson-Hodge conjecture (which generalizes the Hodge conjecture to higher $K$-theory); and
\item it also follows from the arithmetic Bloch-Beilinson conjecture (on injectivity of the Deligne cycle-class), which leads to a proof when $\B\subset \PP^1$ is the complement of the discriminant locus in the base of a Landau-Ginzburg (LG-) model mirror to a Fano 3-fold of Picard rank 1.
\end{itemize}
We will make use of the results of \cite[\S3]{Ke}. The key is to view Gross-Zagier as a statement about the special values of a higher normal function, which is implied whenever the normal function is motivic (comes from a higher cycle).  

Moreover, a version of the conjecture (the one proved in the second bullet point above) makes sense for more general quotients of the upper half-plane, in particular when we extend $\Gamma$ to $\hat{\Gamma}\leq \PSL_2(\RR)$ by a normalizing Fricke-type involution $\gamma_N$. In the statement, $S_4(\Gamma)$ is replaced by $S_4(\hat{\Gamma},\chi)$ where $\chi$ is the character sending $\Gamma\mapsto 1$ and $\gamma_N\mapsto -1$, and the Green's function is anti-invariant under $\gamma_N$ so only descends to a double-cover of $\hat{\Gamma}\backslash\UH$.  This is natural from our point of view, since the above-mentioned LG-models are families of ${K3}$ surfaces of Shioda-Inose type, which in most cases live over a modular curve of the form $\Gamma_0(N)^{+N}\backslash\UH$.  

The modularity of mirror maps for families of rank 19 ${K3}$ surfaces was characterized by Doran \cite{Dor}, and their role as compactified fibers of LG-models mirror to Fano threefolds was shown by Golyshev \cite{Go} in Picard rank 1.  The full classification of the families of lattice polarized ${K3}$ surfaces comprising the LG-model mirrors to Fano threefolds of higher rank was recently completed by Doran and collaborators in  \cite{DHK+}.  We also address here several additional examples of families of rank 19 ${K3}$ surface fibers of  LG-models (first explored by Galkin \cite{Ga1,Ga2}) lying in Noether-Lefschetz loci of these multiparameter families mirror to Fano threefolds of Picard rank greater than one.

\begin{thx}
CFD acknowledges the support of the Natural Sciences and Engineering Council of Canada (NSERC) and the Center of Mathematical Sciences and Applications of Harvard University (CMSA).  MK was supported by NSF Grant DMS-2502708 and the Simons Foundation.  We thank A.~Corti, M.~Elmi, I.~Gaiur, and V.~Golyshev for helpful discussions.
\end{thx}

\section{Beilinson-Hodge implies Gross-Zagier}\label{S2}

Recall that an admissible normal function on a quasi-projective complex manifold $S$ is an extension in the category of admissible VMHS of a Tate object by a pure VHS $\V$.  (For reference the reader may consult \cite[\S2]{KP} and references therein; or see \cite[\S4.2]{GKS} and \cite[\S5]{GK}.)  Via the topological invariant
\begin{equation*}
\ANF_S(\V(p))=\Ext^1_{\AVMHS}(\ZZ(-p),\V)\overset{[\,\cdot\,]}{\underset{\cong}{\longrightarrow}}\Hom_{\MHS}(\ZZ(-p),H^1(S,\V))
\end{equation*}
this is equivalent to a space of Hodge classes.  For simplicity we will write, for any MHS $M$, $\Hg^p(M):=\Hom_{\MHS}(\ZZ(-p),M)$ and $J^p(M):=\Ext^1_{\MHS}(\ZZ(-p),M)$.

When $S$ is a curve and $U\subseteq\overline{S}$ a Zariski open subset, we also denote by $\ANF_U(\V(p))$ the subspace of normal functions nonsingular along $U\setminus S$, so that
\begin{equation*}
\ANF_U(\V(p))	\overset{[\,\cdot\,]}{\underset{\cong}{\longrightarrow}}\Hg^p(\IH^1(U,\V)).
\end{equation*}
Note that unless $U$ is complete, $\IH^1(U,\V)$ is typically not pure.

\subsection{Cycles on Kuga 3-folds}

Let $\Gamma\leq \SL_2(\ZZ)$ be a congruence subgroup not containing $-\text{Id}$.  We are interested in the setting where $\X\overset{\pi}{\to}\B$ is $\E^2\to \Gamma\backslash\UH$, the self-fiber square of the elliptic modular surface, and $\V\subset R^2\pi_*\ZZ$ is the rank-3 generically-transcendental sub-VHS (isomorphic to $\mathrm{Sym}^2\H^1$ of $\E\to\B$).  

Let $\Sigma\subset\B$ be a finite set of CM points and $X_{\Sigma}=\pi^{-1}(\Sigma)$, and $\Xb\to\Bb$ denote the Shokurov compactification.  We have a localization diagram with rows exact sequences of MHS
\begin{equation*}
\xymatrix@C=0.9em{0 \ar [r] & \IH^1(\Bb,\V)\ar [r] \ar @{^(->}[d] & \IH^1(\Bb\setminus\Sigma,\V) \ar [r] \ar @{^(->}[d]& \V|_{\Sigma}(-1) \ar [r] \ar @{^(->}[d] & 0 \\ H^1(X_{\Sigma})(-1) \ar [r] & H^3(\Xb) \ar [r] & H^3(\Xb\setminus X_{\Sigma}) \ar [r] & H^2(X_{\Sigma})(-1) \ar[r] & H^4(\Xb).}	
\end{equation*}
A portion of the $\mathrm{RHom}_{\MHS}$ exact sequence associated to the top row is
\begin{equation}\label{e2.1}
\xymatrix{\ANF_{\Bb\setminus\Sigma}(\V(2)) \ar @{=}[d] \ar [rd]^{\text{sing}_{\Sigma}} \\ \to \Hg^2(\IH^1(\Bb\setminus\Sigma,\V))\ar [r] & \Hg^1(\V|_{\Sigma}) \ar [r]^{\alpha\mspace{40mu}} & J^2(\IH^1(\Bb,\V))\to}
\end{equation}
and so any configuration of divisor classes on $X_{\Sigma}$ with trivial Abel-Jacobi image under $\alpha$ lifts to a normal function.

Now $\IH^1(\Bb,\V)$ is a pure HS of type $(2,0)+(0,2)$, with $S_4(\Gamma)\overset{\cong}{\to}\IH^1(\Bb,\V)^{2,0}_{\CC}$ by $F\mapsto F(\tau)\,dz_1{\wedge} dz_2{\wedge} d\tau$ and identifying $\IH^1$ with $L^2$-cohomology.  (See \cite{Gor,Sc} for the case $\Gamma=\Gamma(N)$, from which the general case follows.) Moreover, if $\tT_m\in \CH^3(\Xb\times\Xb)$ denote the lifts of Hecke correspondences as in \cite[\S4]{Sc}, then any level $4$ relation $\tT=\sum a_m\tT_m$ kills $\IH^1(\Bb,\V)$.  Fix a CM point $\b\in\B$, with  $Z_{\b}\in\CH^1(X_{\b})$ the corresponding primitive CM cycle, in the sense that $[Z_{\b}]$ generates $\V_{\b}\cap\Hg^1(H^2(X_{\b}))\cong \ZZ(-1)$.  Writing $\Sigma:=|\T^*\{\b\}|$ as in \S\ref{S1}, we have $\Z:=\tT^*Z_{\b}\in\CH^1(X_{\Sigma})$ with class in $\ker(\alpha)\subseteq\Hg^1(\V|_{\Sigma})$ by functoriality of Abel-Jacobi.

According to the Beilinson-Hodge conjecture, an integer multiple of the corresponding normal function $\nu\in\ANF_{\Bb\setminus\Sigma}(\V(2))$ should arise from some $\W\in\CH^2(\Xb\setminus X_{\Sigma},1)$.  This $\nu_{\W}$ takes values in the noncompact generalized Jacobians
\begin{equation*}
\nu_{\W}(b)\in J(\V(2)|_b)\cong\frac{\V_{\CC}|_b}{F^2\V_{\CC}|_b+\V_{\ZZ(2)}|_b},
\end{equation*}
given by applying the integral regulator (generalized Abel-Jacobi) to $\W_b$.  Going modulo $\V_{\RR(2)}$ and dividing by $2\pi\ay$ yields a real-analytic section $\r_{\W}$ of $\V_{\RR}^{1,1}$ on $\B\setminus\Sigma$, called the real regulator.  Note that $\W$ may be assumed to be defined over $\QB$ by a spread argument.

\subsection{Geometric Gross-Zagier}\label{S2.2}

There is another ``canonical'' section $\eta$ over $\B$, given by descending
\begin{equation*}
\eta_{\tau}=\frac{dz_1\wedge d\overline{z}_2+d\overline{z}_1\wedge dz_2}{4\Im(\tau)}
\end{equation*}
to fibers of $\pi$, where $\langle\eta_b,\eta_b\rangle:=\int_{X_b}\eta_b\wedge\eta_b=-\frac{1}{2}$ for all $b$. The real-analytic function $\G_{\W}:=\langle\r_{\W},\eta\rangle$ is then, according to Lemmas 3.3, 3.5, and 3.6(b) of \cite{Ke}, a higher Green's function of weight $4$ and level $\Gamma$ with poles along $\Sigma$.  In fact, if for $\sigma\in |\T_m^*\{\b\}|\subseteq\Sigma$ we let ${}^{\sigma}\tT=a_m{}^{\sigma}\tT_m$ denote the restriction of $\tT$ to a correspondence in $\CH^2(X_{\b}\times\X_{\sigma})$, then since ${}^{\sigma}\tT_{m*}\eta_{\b}=m\eta_{\sigma}$ we have
\begin{equation*}
c_{\sigma}=\langle\text{sing}_{\sigma}(\nu_{\W}),\eta_{\sigma}\rangle=\langle{}^{\sigma}\tT^*[Z_{\b}],\eta_{\sigma}\rangle=\langle[Z_{\b}],{}^{\sigma}\tT_*\eta_{\sigma}\rangle=ma_mc_{\b}\,(\in \QQ^{\times}),
\end{equation*}
so that $\G_{\W}=c_{\b}\sum ma_mT_m^*G_{\b}=c_{\b}G_{\T,b}$ by uniqueness.\footnote{This uniqueness corresponds to the fact that higher normal functions with values in an irreducible VHS are determined up to torsion by their singularities \cite[(5.4)]{GK}.}

Finally, \cite[Lem.~3.4]{Ke} shows that $\G_{\W}$ takes values in $\QB\log\QB$ at CM points.  This uses the fact that the restriction of $\W$ to $X_b$ for \emph{any} $b\in \B\setminus\Sigma$ takes the form $\sum(C_i,f_i)$ with $\sum(f_i)=0$, and $\r_{\W}$ is represented by $\sum \log|f_i|\delta_{C_i}$.  When $b$ is CM, these functions and curves are defined over $\QB$, while $[\eta_b]=\sum \gamma_jD_j$ for some $\gamma_j\in \QB$.  We thus get $\G_{\W}(b)=\sum_{i,j}\sum_{p\in C_i\cap D_j}\gamma_j\log|f_i(p)|$ as required by the Gross-Zagier conjecture.

Let us say that the \emph{geometric Gross-Zagier conjecture} holds in weight 4 and level $\Gamma$ if all Green's functions of the form $G_{\T,\b}$ are (up to a rational multiple) ``motivated'' by a cycle $\W$.  Then we have just shown:

\begin{thm}\label{t2.1}
The Beilinson-Hodge conjecture implies the geometric Gross-Zagier conjecture in weight 4 and any level.
\end{thm}

The basic point here is that the normal function is a family of extensions of $\ZZ(-2)$ (type (2,2)) by $\V_b$ (of weight $2$).  When $b$ is CM, a $\V_b$ splits into $\ZZ(-1)$ and a complement, and so the extension has a ``Kummer'' restriction in $\Ext_{\MHS}(\ZZ(-2),\ZZ(-1))$.  If this comes from a cycle defined over $\QB$, then the extension class is the logarithm of an algebraic number.

\begin{rem}
One may extend Theorem \ref{t2.1} to all weights by a similar argument with $\X=\E^{2k-2}$.
\end{rem}

Returning to diagram \eqref{e2.1}, the geometric Gross-Zagier conjecture asks for a cycle $\W$ whose associated normal function has singularity class $[\Z]$.  According to \cite[Prop.~4.1]{Sa}, we have $\text{sing}_{\sigma}(\W)=[\mathrm{Res}_{X_{\sigma}}(\W)]$ for each $\sigma\in\Sigma$.  So it is enough to lift $\tT^*Z_{\b}$ in the diagram
\begin{equation}\label{e2.2}
\xymatrix{\to \CH^2(\Xb\setminus X_{\Sigma},1)_{\QQ} \ar [r]^{\mspace{50mu}\Res_{X_{\Sigma}}} & \CH^1(X_{\Sigma})_{\QQ} \ar [r] \ar @{^(->} [d] & \CH^2(\Xb)_{\QQ} \ar [d]^{c_{\D}} \to \\ & H^2_{\D}(X_{\Sigma},\QQ(1)) \ar [r]^{\tilde{\alpha}} & H^4_{\D}(\Xb,\QQ(2))}
\end{equation}
where we interpret the top row as a portion of the localization sequence \emph{for varieties defined over $\QB$}.  (That is, it consists of algebraic cycles defined over $\QB$.) Since $[\Z]$ already goes to zero in $H^4_{\D}(\Xb,\QQ(2))$ by the analysis in \S\ref{S2}, we see that injectivity of $c_{\D}$ will suffice to produce $\W$:

\begin{cor}
The arithmetic Bloch-Beilinson conjecture for $\Xb$ implies the geometric Gross-Zagier conjecture in weight $4$ and level $\Gamma$.
\end{cor}

Though this injectivity is \emph{conjectured} for any smooth quasi-projective $\Xb$ over $\QB$, we only know it in special cases.  (Here we should remind the reader that $\Xb$ depends on $\Gamma$.) For instance, if $\Xb$ is a \emph{rational} threefold, it follows from the projective bundle formula.  Indeed, rationality of $\Xb$ is shown in \cite[App.~D]{BF} for $\Gamma=\SL_2(\ZZ)$; however, we shall take a slightly different approach in the next section to get a larger collection of examples.

\section{Mirror families of $K3$ surfaces}\label{S3}

Let $\Y\overset{f}{\to}\C$ be a family of (generically) Picard rank 19 $K3$ surfaces over a curve, and $\Yb\to\Cb$ a compactification with smooth total space.  Write $\H^2=\H^2_{\text{tr}}\oplus\H^2_{\text{alg}}$ for the orthogonal decomposition of $R^2f_*\QQ$ (as a VHS) into generically transcendental and generically algebraic parts.

\subsection{Modular families}\label{S3.1}
We first make precise the version of Gross-Zagier alluded to at the end of the introduction.

\begin{defn}
The $K3$ family is said to be \emph{modular} if there exist:
\begin{itemize}[leftmargin=0.5cm]
\item a congruence subgroup $\Gamma\leq\SL_2(\ZZ)$, with corresponding Kuga threefold $\Xb\to\Bb$ and $\QQ$-VHS $\V$ over $\B$ (as in \S\ref{S2}, but $\otimes\QQ$);
\item a finite extension $\hat{\Gamma}\leq\PSL_2(\RR)$ of $\PP\Gamma$, such that $\C\subset\hat{\Gamma}\backslash\UH$ is the complement of the orbifold points; and
\item a correspondence $\Theta\in\CH^3(\Xb\times\Yb)$ over the natural quotient map $\overline{\theta}\colon\Bb\to\Cb$, whose action on cohomology induces an isomorphism $\theta^*\H^2_{\text{tr}}\cong\V$.
\end{itemize}
More precisely, there is a family $\K\to\C$ of Kummer $K3$s associated to $\Y$, and a rational map $\Yb\overset{\kappa}{\dashrightarrow}\Kb$ over the identity on $\Cb$, and a rational map $\Xb\overset{\tilde{\Theta}}{\dashrightarrow}\Kb$ over $\theta$ ``undoing the Kummer construction'' as in \cite[\S5]{DHNT} together with a fiberwise isogeny; our $\Theta$ is the ``composition'' of $\tilde{\Theta}$ with the inverse of $\kappa$.
\end{defn}

In all the examples we know, $\Y$ is a family of Shioda-Inose $K3$s and the generically 2:1 map $\kappa$ quotients by the Nikulin involution and resolves singularities.

By a \emph{Fricke-type involution}, we will mean an element of $\PSL_2(\RR)$ which is $\PSL_2(\ZZ)$-conjugate to $W_N=\begin{psmatrix}0&1/\sqrt{N}\\-\sqrt{N}&0\end{psmatrix}$ for some $N$.

\begin{defn}
We will call a modular family \emph{special modular} if
\begin{itemize}[leftmargin=0.5cm]
\item $\hat{\Gamma}$ has a normal subgroup $\Gamma_0\leq \PSL_2(\ZZ)$, with quotient $\g:=\hat{\Gamma}/\Gamma_0\cong(\ZZ/2\ZZ)^{\ell}$ generated by $\ell$ Fricke-type involutions;
\item sections of the extended Hodge bundle $\H^{2,0}_{\text{tr},e}$ correspond to $M_2(\hat{\Gamma},\chi)$ and holomorphic $3$-forms on $\Yb$ to $S_4(\hat{\Gamma},\chi)$, where $\chi\colon \g\to\{\pm 1\}$ sends the involutions to $-1$;\footnote{$\Gamma_0$ plays the role of $\Gamma$ in the last paragraph of \S\ref{S1}.} and
\item the base-change $\V_0$ of $\H^2_{\text{tr}}$ to $\B_0=\Gamma_0\backslash\UH$ admits a section $\eta_0$ of $(\V_0)^{1,1}_{\RR}$ with $\langle\eta_0,\eta_0\rangle\equiv-\frac{1}{2}$, on which $\mathrm{Aut}(\B_0/\C)\cong \g$ acts through $\chi$.
\end{itemize}
\end{defn}

In practice $\ell$ is usually $1$, but we at least need to allow for it to be $2$.

\begin{example}
Group data occurring for $K3$ families in \S\ref{S3.3} include:
\begin{itemize}[leftmargin=0.7cm]
\item [(a)] $\hat{\Gamma}=\Gamma_0(N)^{+N}$ is generated by $\Gamma_0=\Gamma_0(N)$ and $W_N$, which always normalizes $\Gamma_0(N)$.\footnote{$\Gamma_0(N)^{+N}$ is denoted $\Gamma_0^*(N)$ in some literature (e.g.~\cite{GoZ}).}
\item [(b)] $\Gamma_0=\Gamma_0(6)$ is normalized by $\beta_3=\begin{psmatrix}\sqrt{3} & -\frac{1}{\sqrt{3}} \\ 4\sqrt{3} & -\sqrt{3}\end{psmatrix}$ (but \emph{not} $W_3$), which is the conjugate of $W_3$ by $\begin{psmatrix}1&0\\-3&1\end{psmatrix}$;  by $(\hat{\Gamma}=)\,\Gamma_0(6)^{+3}$ we mean the group generated by $\Gamma_0$ and $\beta_3$.  We can also take $\hat{\Gamma}=\Gamma_0(6)^{+3+6}$ to be the group generated over $\Gamma_0$ by $\beta_3$ and $W_6$; this gives $\ell=2$.
\item [(c)] Let $\widetilde{\Gamma}_0(8)\,(=\Gamma_0)$ denote the conjugate of $\Gamma_0(8)$ by $\begin{psmatrix}1&-1\\2&-1\end{psmatrix}$; this is normalized by $W_4$ and together they generate $\hat{\Gamma}=\widetilde{\Gamma}_0(8)^{+4}$.
\end{itemize}
For $\Gamma_0=\Gamma_0(N)$, we can always take $\Gamma=\Gamma_0(2N)\cap \Gamma(2)$.
\end{example}

Now consider a CM point $\b_0\in \B_0$ with image $\c\in \C$, together with a relation $\T$ on $S_4(\hat{\Gamma},\chi)$.  The higher Green's function $\hat{G}_{\T,\c}:=\sum_{\gamma\in \g} \chi(\gamma)\gamma^*G_{\T,\b_0}$ is anti-invariant under the Fricke-type involutions, and descends only to the quotient of $\B_0$ by $\ker(\chi)\leq \g$.  Write $\Sigma\subset\C$ for the image of $|\T^*\{\b_0\}|$ in $\C$.

The \emph{Gross-Zagier conjecture in weight $4$ and level $\hat{\Gamma}$} says that $\hat{G}_{\T,\c}$ takes $\QB\log\QB$ values at CM points outside its singularities. This is implied by the corresponding geometric Gross-Zagier conjecture, which says that $\hat{G}_{\T,\c}$ arises (up to a rational multiple) from a cycle $\fW\in\CH^2(\Yb\setminus Y_{\Sigma},1)$ by base-changing the latter to $\B_0$ and pairing its real regulator with $\eta_0$.

For the rest of this note we will be concerned only with the case where $S_4(\Gamma_0,\chi)=\{0\}$, so that no relation $\T$ is needed and $\hat{G}_{\c}=\sum_{\gamma\in \g}\chi(\gamma)G_{\gamma(\b_0)}$ is the Green's function we want to realize by a higher cycle.  In most cases this takes the form $G_{\b_0}-G_{\gamma(\b_0)}$.

\subsection{Cycles on $K3$ families}\label{S3.2}

Setting the contents of \S\ref{S3.1} aside for the moment, suppose that $\Cb\cong\PP^1$ (with coordinate $t$) and $\Yb$ is a \emph{rational} 3-fold.  Suppose that $t_0\in\C$ is a CM point (i.e.~$Y_{t_0}$ has Picard rank 20) and that $\fZ_{t_0}\in\CH^1(Y_{t_0})$ has class generating  $\Hg^1(\H^2_{\text{tr}}|_{t_0})\cong\ZZ$.

We want to argue that a multiple of the CM cycle $\imath^{t_0}_*\fZ_{t_0}$ vanishes in $\CH^2(\Yb)$.  \emph{For this we assume additionally that $H^3(\Yb)=\{0\}$.}  Then $\CH^2(\Yb)_{\QQ}$ is isomorphic to the space of $\QQ$-Hodge classes in $H^4(\Yb)$.  In the diagram
\begin{equation*}
\xymatrix{\CH^2(\Yb\setminus Y_{t_0},1)_{\QQ} \ar [r]^{\mspace{20mu}\Res_{Y_{t_0}}} & \CH^1(Y_{t_0})_{\QQ} \ar [r]^{\imath^{t_0}_*} \ar [d]^{\cong}_{[\,\cdot\,]} & \CH^2(\Yb)_{\QQ} \ar  [d]^{[\,\cdot\,]}_{\cong} \\ & \Hg^1(Y_{t_0})_{\QQ} \ar [r]^{\imath^{t_0}_*} & \Hg^2(\Yb)_{\QQ}}
\end{equation*}
we claim that $\imath^{t_0}_*[\fZ_{t_0}]$ is zero in $\Hg^2(\Yb)_{\QQ}$.  Suppose otherwise: then there is a dual Hodge class $[\D]\in\Hg^1(\Yb)_{\QQ}$, with $1=\langle[\D],\imath^{t_0}_*[\fZ_{t_0}]\rangle_{\Yb}=\langle\imath_{t_0}^*[\D],[\fZ_{t_0}]\rangle_{Y_{t_0}}$. But  $\imath^*_{t_0}[\D]$ remains algebraic under parallel translation to nearby non-CM fibers (viz., $\imath^*_{t_1}[\D]$), hence belongs to $\H^2_{\text{alg}}|_{t_0}$ and must be orthogonal to $[\fZ_{t_0}]\in\H^2_{\text{tr}}|_{t_0}$, a contradiction. Thus $\imath^{t_0}_*\fZ_{t_0}=0$ and we obtain a cycle $\fW\in\CH^2(\Yb\setminus Y_{t_0},1)$ bounding (via $\Res_{Y_{t_0}}$) on a multiple of $\fZ_{t_0}$.  

Next, assume only that $\Yb$ is rational and special modular.  Then the argument of the last paragraph still gives that $\imath^{t_0}_*[\fZ_{t_0}]=0$, and the argument around diagram \eqref{e2.2} reduces triviality of $\imath^{t_0}_*\fZ_{t_0}$ to that of its Abel-Jacobi image in $J^2(\IH^1(\Cb,\H^2_{\text{tr}}))$.  But the parabolic cohomology $\IH^1(\Cb,\H^2_{\text{tr}})$ is of $(3,0)+(0,3)$ type by virtue of injecting into $\IH^1(\Bb,\V)$ under $\Theta^*$, and of $(1,2)+(2,1)$ type by virtue of rationality.  So it is zero, and we get $\fW$ as before, leading to the

\begin{thm}
If $\Yb\to\Cb\subset\hat{\Gamma}\backslash\UH$ is a special modular family of $K3$ surfaces with rational total space, then the geometric Gross-Zagier conjecture in weight 4 and level $\hat{\Gamma}$ holds.
\end{thm}
\begin{proof}
Let $\hat{\G}_{\fW}$ be the real-analytic function defined by pairing the pullback of $\r_{\fW}$ to $\B_0$ with the section $\eta_0$ of $(\V_0)^{1,1}_{\RR}$.  Applying the first paragraph of \S\ref{S2.2} to $\Theta^*\fW$ shows that its pullback to $\B$ is a higher Green's function; so this is true on $\B_0$ as well.  Since this has the same singularities and behavior under $\g$ as $\hat{G}_{t_0}$, they are equal (up to a rational multiple) by uniqueness.
\end{proof}

\subsection{The modular mirror families}\label{S3.3}

Let $\vf\in \ZZ[x^{\pm1},y^{\pm1},z^{\pm1}]$ be a Minkowski polynomial with (reflexive) Newton polytope $\Delta\subset\RR^3$ \cite{CC+1}.  The level hypersurfaces $Y_t^*:=\{1-t\vf=0\}\subset\GG_m^3$ have $K3$ compactifications $Y_t$ for general $t$, with boundary $Y_t\setminus Y_t^*$ a union of smooth rational curves.  Such $\vf$ are classified in \cite{CC+2}, with the resulting $H^2_{\text{tr}}(Y_t)$ giving 165 distinct VHS, 23 of which have rank $3$ (so that $Y_t$ has Picard rank 19 for very general $t$).  Of these, $15$ arise as mirrors of minimal (i.e., Picard rank $1$) Fano 3-folds, with the polynomials $\vf$ listed in \cite[Table 1]{HLP}; while 8 more are mirror to ``$G$-minimal'' Fano 3-folds of Picard ranks $2$, $3$, and $4$ \cite{Ga1,Ga2}.

By $\PD$, we shall mean the smooth toric 3-fold associated to a maximal projective triangulation of the polar $\Delta^{\circ}$. The Zariski closures $\overline{Y^*_t}\subset\PD$ intersect in a base locus described by the polynomials $\vf_F$, as $F$ ranges over faces of $\Delta$. Since $\vf$ is a Minkowski polynomial, this base locus is a union of smooth rational curves.
The key point is that the total space $\Yb$ is constructed from $\PD$ by iteratively blowing up along (successive proper transforms of) components of this locus.  Hence $\Yb$ is rational.

Furthermore, as partially documented in \cite{Go,GoZ} and \cite{Ga1,Ga2}, the 23 families just described are special modular.  The period map presents each $\H^2_{\text{tr}}$ as a finite basechange of the canonical VHS over $\Gamma_0(N)^+\backslash\UH$ for some $N$, where $\Gamma_0(N)^+$ is the \emph{full} normalizer of $\Gamma_0(N)$ in $\PSL_2(\RR)$.  Outside of $9$ cases already done in \cite{GoZ} (with $\hat{\Gamma}=\Gamma_0(N)^{+N}$), the main difficulty is in checking that this cover is of the form $\hat{\Gamma}\backslash\UH$, and that it has a cover of the form $\Gamma_0\backslash\UH$ over which $\H^2_{\text{tr}}$ admits a weight-one ``symmetric square root''.

\begin{rem}
The cases $(N,d)=(1,1)$ and $(1,2)$ in \cite{GoZ} do not appear in our table because the 2:1 covers given by $\sqrt{t}$ and $\sqrt{1-1728t}$ both involve an order-2 ramification at an order-3 orbifold point on the $j$-line, so are not quotients of $\UH$.  Though the relevant $K3$ families are not given by a Minkowski polynomial (the Newton polytopes of $\vf=\tfrac{(1+x+y+z)^6}{xyz}$ and $\tfrac{(1+x+y)^6}{xy^2z}+z$ from \cite{HLP} are not reflexive), they still have rational total spaces and the construction of $\fW$ in \S\ref{S3.2} still works.  But strictly speaking, the interpretation of the resulting real-analytic function on a 2:1 or 4:1 cover of the $j$-line as a higher Green's function doesn't work.
\end{rem}

In Table 1 below, we record the Fano variety $\mathrm{X}$ to which each $K3$ family is mirror, the Laurent polynomial $\vf$, and the subgroup $\hat{\Gamma}$ of $\PSL_2(\RR)$.  To obtain $\Gamma_0$ from the latter, just remove the ``${+}N$'' or ``${+}N{+}M$'' exponent.  The notation for groups comes from the database \cite{CP}; something like ``$12\text{F}^1$'' means that the group has level $12$ and $12\text{F}^1\backslash\UH$ has genus $1$.

\begin{table}[h!]\label{tab1}
    \centering
    \begin{tabular}{|c|c|c|}
\hline
      $\mathrm{X}$ & $\vf$  & $\hat{\Gamma}$  \\
     \hline
        (2,1) & $\tfrac{(1+x+y+z)^4}{xyz}$ & $\Gamma_0(2)^{+2}$ \\
     \hline
        (3,1) & $\tfrac{(1+x)^2(1+y+z)^3}{xyz}$ & $\Gamma_0(3)^{+3}$ \\
     \hline
        (4,1) & $\tfrac{(1+x)^2(1+y)^2(1+z)^2}{xyz}$ & $\Gamma_0(4)^{+4}$ \\
     \hline
        (5,1) & $\tfrac{(1+x+y+z+xy+xz+yz)^2}{xyz}$ & $\Gamma_0(5)^{+5}$ \\
     \hline
        (6,1) & $\tfrac{(1+x+z)(1+x+y+z)(1+z)(y+z)}{xyz}$ & $\Gamma_0(6)^{+6}$ \\
     \hline
        (7,1) & $\tfrac{(1+x+y+z)^2}{x}+\tfrac{(1+z)^2(1+y+z)(1+x+y+z)}{xyz}$ & $\Gamma_0(7)^{+7}$ \\
     \hline
        (8,1) & $\tfrac{(1+x)(1+y)(1+z)(1+x+y+z)}{xyz}$ & $\Gamma_0(8)^{+8}$ \\
     \hline
        (9,1) & $\tfrac{(x+y+z)(x+xz+xy+xyz+z+y+yz)}{xyz}$ & $\Gamma_0(9)^{+9}$ \\
     \hline
        (11,1) & $3+z+\tfrac{xy}{z}+\tfrac{(1+z)(1+x+y)(xy+z)}{xyz}$ & $\Gamma_0(11)^{+11}$ \\
     \hline
        (2,2) & $z+\tfrac{(1+x+y)^4}{xyz}$ & $(4\text{C}^0)^{+2}$ \\
     \hline
        (3,2) & $z+\tfrac{(1+x+y)^3}{xyz}$ & $(6\text{C}^0)^{+3}$ \\
     \hline
        (4,2) & $z+\tfrac{(1+x)^2(1+y)^2}{xyz}$ & $\tilde{\Gamma}_0(8)^{+4}$ \\
     \hline
        (5,2) & $x+y+z+\tfrac{1}{x}+\tfrac{1}{y}+\tfrac{1}{z}+xyz$ & $(10\text{A}^1)^{+5}$ \\
     \hline
        (3,3) & $z+y+\tfrac{(1+x)^2}{xyz}$ & $(9\text{A}^1)^{+3}$ \\
     \hline
        (2,4) & $x+y+z+\tfrac{1}{xyz}$ & $(8\text{A}^1)^{+2}$ \\
     \hline
        2-6 & see \#3873.2 in \cite{CC+2} & $(3\text{C}^0)^{+1}$? \\
     \hline
        2-12 & see \#1193 in \cite{CC+2} & $\Gamma_0(10)^{+10+5}$? \\
     \hline
        2-21 & $x+y+z+\tfrac{x}{z}+\tfrac{y}{z}+\tfrac{x}{y}+\tfrac{y}{x}+\tfrac{z}{y}+\tfrac{1}{x}+\tfrac{1}{z}$ & $\Gamma_0(14)^{+14+7}$? \\
     \hline
        2-32 & $x+y+z+\tfrac{1}{x}+\tfrac{1}{y}+\tfrac{1}{xyz}$ & $(6\text{B}^1)^{+1}$? \\
     \hline
        3-1 & see \#3873.4 in \cite{CC+2} & $\Gamma_0(6)^{+6+3}$ \\
     \hline
        3-13 & $x+y+z+\tfrac{y}{z}+\tfrac{x}{y}+\tfrac{z}{x}+\tfrac{1}{x}+\tfrac{1}{y}+\tfrac{1}{z}$ & $\Gamma_0(15)^{+15+5}$? \\
     \hline
        3-27 & $x+y+z+\tfrac{1}{x}+\tfrac{1}{y}+\tfrac{1}{z}$ & $(12\text{F}^1)^{+6+3}$ \\
     \hline
        4-1 & $(1+x+y+z)(1+\tfrac{1}{x}+\tfrac{1}{y}+\tfrac{1}{z})$ & $\Gamma_0(6)^{+3}$ \\
     \hline
    \end{tabular}\normalsize
    \caption{23 mirror special modular $K3$ families.}
    \label{tab_2}
\end{table}

For $\mathrm{X}$, we give either the $(N,d)$ data (the \emph{index} $d$ of $-K_{\mathrm{X}}$ in $H^2(\mathrm{X},\ZZ)$ and the \emph{level} $N=(-K_X)^3/2d^2$) or the Mori-Mukai data ($\rho$-$**$ where $\rho$ is the Picard rank and $**$ its position in the list).  It seems standard to call the $(N,1)$-Fanos\footnote{These are described explicitly in several of the references.  Note that the mirror of $V_{12}$ is the family of $K3$s underlying Ap\'ery's irrationality proof for $\zeta(3)$ (e.g.~see \cite{Ke2}); mirrors of $V_{10}$, $V_{12}$, $V_{14}$, $V_{16}$, and $V_{18}$ have been studied in detail in \cite{GKS}.} $V_{2N}$ and the $(N,2)$-Fanos $\mathcal{B}_N$, while $(3,3)$ and $(2,4)$ are the quadric and $\PP^3$ respectively; in the Mori-Mukai list lie familiar varieties like $\mathrm{MM}_{3\text{-}27}=(\PP^1)^3$ and $\mathrm{MM}_{4\text{-}1}=$ the hypersurface of multidegree (1,1,1,1) in $(\PP^1)^4$.  See \cite{CCGK} for a full accounting.  

The Laurent polynomials are taken from \cite{HLP,CC+2} and are \emph{not unique}.  Morally it is the LG-model $\overline{f}\colon \Yb\to\PP^1_t$ mirror to $\mathrm{X}$ that is unique, and for each torus chart $\GG_m^3\overset{\rho}{\to}\Yb\setminus Y_0$ the restriction $(1/\overline{f})\circ \rho$ defines a Laurent polynomial $\vf$ in the torus coordinates.  Each torus chart is mirror to a toric specialization $\mathrm{X}\rightsquigarrow\PP_{\Delta^{\circ}}$, with $\Delta^{\circ}$ the polar of the Newton polytope of $\vf$.  See \cite{Co} for a more detailed explanation.

The upshot is of course the following

\begin{cor}
The geometric Gross-Zagier conjecture holds\footnote{The question marks in the table are educated guesses, and we do not have full proofs in those cases.  So this result is conjectural in those 5 cases.} in weight $4$ and level $\hat{\Gamma}$ for each of the 23 groups appearing in Table 1.
\end{cor}

We conclude with some more detailed explanations in several cases.  Note that in \emph{every} case, $t=0$ is a maximal unipotent monodromy point, and  the discriminant locus takes the form $\{0\}\cup \mathcal{A}$.

\begin{example}
(a) In the first 9 cases, $\mathcal{A}$ consists of one finite point and $\infty$ ($N=2,3,4$), two finite points ($N=5,6,7,8,9$), or four finite points ($N=11$) \cite{GoZ}.  Order three Picard-Fuchs ordinary differential equations and three-dimensional integral monodromy representations in each of these cases are computed in \cite{DoKo} using the Griffiths-Dwork algorithm (op cit., Table 7) and geometric variation of local systems approach (op cit., Table 8), respectively. The 2:1 cover of $\Cb\cong\PP^1$ branched along $\mathcal{A}$ is $\Bb_0^{(N,1)}$; this ``undoes'' the quotient by the Fricke involution $W_N$ and supports a family of elliptic curves.  It has genus 0 except for $N=11$.

(b) In the next six cases, the $(N,d)$ $K3$ family is the basechange of the $(N,1)$ $K3$ family by $t\mapsto t^d$.  The fiber product of the Fricke cover and this cyclic one yields $\Bb_0^{(N,d)}$, which has genus 0 and genus 1 in three cases each.  The groups in this part of the table are computed by examining the basechange of the elliptic modular surface over $\Bb_0^{(N,d)}\twoheadrightarrow\Bb_0^{(N,1)}$.  This necessarily produces subgroups of $\Gamma_0(N)$ which are normalized by $W_N$, like $\tilde{\Gamma}_0(8)$ in the $(4,2)$ case.  (In the other cases the group is not conjugate to $\Gamma_0(Nd)$, which has either the wrong index or wrong cusp-widths.)

(c) Family 3-1 has $\mathcal{A}=\{\tfrac{1}{36},\tfrac{1}{4},\infty\}$, with  a simple reflection (conifold monodromy) at the two finite points and monodromy exponents $(\tfrac{1}{2},1,\tfrac{1}{2})$ at $\infty$. Families (6,1) (Ap\'ery), 3-27 (Fermi), and 4-1 (Verrill) are all basechanges of this, by the maps $t\mapsto \tfrac{t}{(t+1)^2}$, $t\mapsto t^2$, and $t\mapsto\tfrac{36t}{(8t+1)^2}$; see \cite[\S3.2]{BKV} as well as \cite[\S10]{DK} and \cite[\S5]{Ke2}.  The first and last of these maps are quotients by the Fricke-type involutions $\beta_3$ and $W_6$, and have $\Gamma_0(6)\backslash\UH$ as a fiber product; the fiber product of all three is $12\text{F}^1\backslash\UH$ (with genus 1 and 4 cusps of widths $2,4,6,12$).
\end{example}

\section{Multiparameter and higher dimensional extensions}\label{S4}

The last 8 $K3$ families in the table are mirror to the restriction of the quantum $D$-module of a Picard rank $r$ Fano 3-fold to the anticanonical 1-torus as in \cite{Go} -- called ``standard LG-models'' in \cite{DHK+}. The full $r$-parameter mirrors or ``parametrized LG-models'' are constructed in [op.~cit.], and (possibly after finite basechange) their $\H^2_{\text{tr}}$ is a Hermitian VHS over a type IV locally symmetric variety. Such modular families of $K3$ surfaces are known to have (countably) many indecomposable $K_1$ cycles on a very general fiber \cite{CDKL}, but this tells us nothing about the finite cover of the family on which they become well-defined. It would be interesting to know how our cycles on (and well-defined on!) the 1-parameter subfamilies extend to the $r$-parameter families: what divisors must be removed from the base, on what finite cover of the base they become well-defined, etc. 

Going beyond modular families, one may consider analogues of higher Green's functions on curves parametrizing even dimensional Calabi-Yau varieties. For instance, let $\Yb\to\PP^1$ be a family of CY 4-folds over $\PP^1$, whose middle cohomology VHS $\H^4$ has all Hodge numbers $1$. When $\Yb$ is extremal (i.e.~$\IH^1(\PP^1,\H^4)$ is trivial), and $t_0$ is a Hodge point (so that $\Hg^2(Y_{t_0})$ is nontrivial), there exists a nontorsion normal function $\nu\in\ANF_{\PP^1\setminus\{t_0\}}(\H^4(3))$. Suppose $t_1$ is another Hodge point. Then if Beilinson-Hodge holds, a cycle in $\fZ\in\CH^3(\Yb\setminus Y_{t_0},1)$ underlies the normal function and a cycle in $\CH^2(Y_{t_1})$ represents the second Hodge class. Passing to a finite cover, there exists a family of normalized real $(2,2)$-forms and we can pair this against the real regulator of $\fZ_t$ to get a ``Green's function'', whose (suitably normalized) special value at $t_1$ is again in $\QB\log\QB$. 

The geometrically natural pencils of (1,1,1,1,1) CY 4-folds arising as level sets of Laurent polynomials (and/or mirror to Fano 5-folds) produce VHSs which are as different from $\mathrm{Sym}^4\H^1$ of an elliptic curve family as possible. With generic Hodge group $\mathrm{SO}(2,3)$, the Hodge points are atypical and expected to be finite in number, in contrast to their abundance in the modular setting. These are closely related to CM and attractor points on families of (1,1,1,1) CY 3-folds by the so-called Yifan-Yang pullback (see \S 3.3.1 of \cite{DM}). Even to locate two Hodge points in many such families may be a challenge, but one worth undertaking with an eye to what such special values might have to say about the arithmetic geography of the Fano 5-fold mirrors.

\curraddr{${}$\\
\noun{Department of Mathematical and Statistical Sciences}\\
\noun{University of Alberta}\\
\noun{Edmonton, AB T6G 2G1, Canada}}
\email{${}$\\
\emph{e-mail}: charles.doran@ualberta.ca}

\curraddr{${}$\\
\noun{Dept. of Mathematics and Statistics, Campus Box 1146}\\
\noun{Washington University in St. Louis}\\
\noun{St. Louis, MO} \noun{63130, USA}}
\email{${}$\\
\emph{e-mail}: matkerr@wustl.edu}

\end{document}